\documentclass{amsart}[10pt]

\usepackage{amssymb,mathrsfs,amscd,graphicx}
\usepackage{abstract}
\usepackage{hyperref}
\usepackage{cite}
\usepackage{enumerate}
\usepackage[perpage,symbol]{footmisc}   

\usepackage[margin=3cm]{geometry}
\setfnsymbol{wiley}
\numberwithin{equation}{section}
\makeatother  

\DeclareMathSymbol{\twoheadrightarrow} {\mathrel}{AMSa}{"10}

\DeclareMathOperator*{\Div}{div}
\DeclareMathOperator*{\tr}{tr}

\DeclareMathOperator*{\sign}{sign}

\DeclareMathOperator*{\Ric}{Ric}

\DeclareMathOperator*{\Hess}{Hess}

\newcounter{thmcounter}
\numberwithin{thmcounter}{section}  

\newtheorem{thm}[thmcounter]{Theorem}
\newtheorem{lem}[thmcounter]{Lemma}
\newtheorem{cor}[thmcounter]{Corollary}

\theoremstyle{definition}

\newtheorem{rem}[thmcounter]{Remark}

\begin{document}

\title[Gradient estimates and Harnack inequalities]{Gradient estimates and Harnack inequalities of a parabolic equation under geometric flow}
\author[G. Zhao]{Guangwen Zhao}
\address{Guangwen Zhao, School of Mathematical Sciences, Fudan University, Shanghai 200433, China}
\email{gwzhao@fudan.edu.cn}
\maketitle

\begin{abstract}
  In this paper, we consider a manifold evolving by a general geometric flow and study parabolic equation
  \[
  (\Delta -q(x,t)-\partial_t)u(x,t)=A(u(x,t)),\quad (x,t)\in M\times [0,T].
  \]
  We establish space-time gradient estimates for positive solutions and elliptic type gradient estimates for bounded positive solutions of this equation. By integrating the gradient estimates, we derive the corresponding Harnack inequalities. Finally, as applications, we give gradient estimates of some specific parabolic equations.
\end{abstract}     


\section{Introduction}

The paper study parabolic equation 
\begin{equation}\label{eqa}
  \begin{split}
    (\Delta -q(x,t)-\partial_t)u(x,t)=A(u(x,t))
  \end{split}
\end{equation}
on Riemannian manifold $M$ evolving by the geometric flow 
\begin{equation}\label{eqb}
  \begin{split}
    \frac{\partial }{\partial t}g(t)=2h(t),
  \end{split}
\end{equation}
where $(x,t)\in M\times [0,T]$, $q(x,t)$ is a function on $M\times [0,T]$ of $C^2$ in $x$-variables and $C^1$ in $t$-variable, $A(u)$ is a function of $C^2$ in $u$, and $h(t)$ is a symmetric $(0,2)$-tensor field on $(M,g(t))$. A important example would be the case where $h(t)=-\Ric (t)$ and $g(t)$ is a solution of the Ricci flow introduced by R.S. Hamilton \cite{hamilton1982Three}. We will give some gradient estimates and Harnack inequalities for positive solutions of equation \eqref{eqa}. 

The study of gradient estimates for parabolic equations originated with the work of P. Li and S.-T. Yau \cite{li1986parabolic}. They prove a space-time gradient estimate for positive solutions of the heat equation on a complete manifold. By integrating the gradient estimate along a space-time path, a Harnack inequality was derived. Therefore, Li--Yau inequality is often called differential Harnack inequality. It is easy to see that the above space-time estimate will become an elliptic type gradient estimate for a time-independent solution (see \cite{cheng1975differential}). But the elliptic type estimate cannot hold for a time-dependent solution in general, this can be seen from the form of the fundamental solution of the heat equation in $\mathbb R^n$. However, in 1993, R.S. Hamilton \cite{hamilton1993matrix} established an elliptic type gradient estimate for positive solutions of the heat equation on compact manifolds. It is worth noting that the noncompact version of Hamilton's estimate is not true even for $\mathbb R^n$ (see \cite[Remark~1.1]{souplet2006sharp}). Nevertheless, for complete noncompact manifolds, P. Souplet and Q.S. Zhang \cite{souplet2006sharp} obtained an elliptic type gradient estimate for a bounded positive solution of the heat equation after inserting a necessary logarithmic correction term. Li--Yau type and Hamilton--Souplet--Zhang type gradient estimates have been obtained for other nonlinear equations on manifolds, see for example \cite{chen2009gradient,chen2016gradient,dung2015gradient,li1991Gradient,li2015li,ma2006gradient,negrin1995Gradient,ruan2007elliptic,wu2015elliptic,wu2017Elliptic,yang2008gradient} and the references therein.

On the other hand, gradient estimates are very powerful tools in geometric analysis. For instance, R.S. Hamilton \cite{hamilton1993The,hamilton1995Harnack} established differential Harnack inequalities for the Ricci flow and the mean curvature flow. These results have important applications in the singularity analysis. Over the past two decades, many authors used similar techniques to prove gradient estimates and Harnack inequalities for geometric flows. The list of relevant references includes but is not limited to \cite{bailesteanu2010Gradient,cao2009Differential,guo2014Harnack,ishida2014Geometric,li2018li,li2016Harnack,liu2009Gradient,ni2004Ricci,sun2011Gradient,zhang2006Some,zhao2016Gradient}. In this paper, we follow the work of J. Sun \cite{sun2011Gradient} and M. Bailesteanu et al. \cite{bailesteanu2010Gradient}, and focus on the system \eqref{eqa}--\eqref{eqb}. 

Now we give some remarks on equation \eqref{eqa}. When $A(u)=au\log u$, the nonlinear elliptic equation corresponding to \eqref{eqa} is related to the gradient Ricci soliton. When $A(u)=au^\beta $, the nonlinear elliptic equation corresponding to \eqref{eqa} is related to the Yamabe-type equation. In general, the parabolic equation \eqref{eqa} is the so-called reaction-diffusion equation, which can be found in many mathematical models in physics, chemistry and biology (see \cite{rothe1984Global,Smoller1983Shock}), where $qu+A(u)$ and $\Delta u$ are the reaction term and the diffusion term, respectively. The reaction-diffusion equations are very important objects in pure and applied mathematics. 

In \cite{chen2018Li}, Q. Chen and the author studied the equation \eqref{eqa} with a convection term on a complete manifold with a fixed metric. Here, we establish some gradient estimates for positive solutions of \eqref{eqa} under geometric flow \eqref{eqb}, which are richer and sharper than \cite{chen2018Li}.

The rest of this paper is organized as follows.

In Section~\ref{sec2}, we establish space-time gradient estimates for positive solution of \eqref{eqa}. We firstly consider that $M$ is a complete noncompact manifold without boundary. A local and a global estimate were established, see Theorem~\ref{tha} and Corollary~\ref{coa}. Next, the case that $M$ is closed is also deal with. In this case, inspired by \cite{bailesteanu2010Gradient}, we obtain a sharper estimate than \cite[Theorem~6]{sun2011Gradient}, see Theorem~\ref{thb}. We also give the corresponding Harnack inequalities in the above two cases, see Corollary~\ref{coc}.

In Section~\ref{sec3}, we consider the case that the solution is bounded, and establish elliptic type gradient estimates of local and global versions, see Theorem~\ref{thc} and Corollary~\ref{cob}. The elliptic type Harnack inequality is also obtained, see Corollary~\ref{cod}.  

Finally, in Section~\ref{sec4}, we give some applications and explanations of these gradient estimates in some specific cases. For the case of $A(u)=au\log u$ with $a\in \mathbb R$, we can derive the gradient estimate for positive solutions. In particular, we deal with the case that the manifold evolving by the Ricci flow, see Corollary~\ref{coe}, \ref{cof}, \ref{cog}. For the case of $A(u)=au^\beta $ with $a\in \mathbb R$ and $\beta \in (-\infty ,0]\cup [1,+\infty)$, we give the gradient estimate for bounded positive solutions, see Corollary~\ref{coh}, \ref{coi}.

Throughout the paper, we denote by $n$ the dimension of the manifold $M$, and by $d(x,y,t)$ the geodesic distance between $x,y\in M$ under $g(t)$. When we say that $u(x,t)$ is a solution to the equation \eqref{eqa}, we mean $u$ is a solution which is smooth in $x$-variables and $t$-variable. In addition, we have to give some notations for the convenience of writing. Let $f=\log u$ and $\hat A(f)=\frac{A(u)}{u}$. Then 
\[
\hat A_f=A'(u)-A(u)/u,\quad \hat A_{ff}=uA''(u)-A'(u)+A(u)/u.
\]
For $u>0$ we define several nonnegative real numbers (some of $\lambda , \Lambda ,\Sigma ,\kappa $ are allowed to be infinite) as follows: 
\begin{align*}
  &\lambda_{2R}:=-\min_{Q_{2R,T}}\hat A_f^-=-\min \left\{0,\min_{Q_{2R,T}}(A'(u)-A(u)/u)\right\},\\
  &\Lambda_{2R}:=\max_{Q_{2R,T}}\hat A_f^+=\max \left\{0,\max_{Q_{2R,T}}(A'(u)-A(u)/u)\right\},\\
  &\Sigma_{2R}:=\max_{Q_{2R,T}}\hat A_{ff}^+=\max \left\{0,\max_{Q_{2R,T}}(uA''(u)-A'(u)+A(u)/u)\right\}\\
  &\kappa_{2R}:=-\min \left\{0,\min_{Q_{2R,T}}\left(A'(u)-A(u)/u\right),\min_{Q_{2R,T}}A'(u)\right\}
\end{align*}
and
\begin{align*}
  &\lambda :=-\inf_{M\times [0,T]}\hat A_f^-=-\min \left\{0,\inf_{M\times [0,T]}(A'(u)-A(u)/u)\right\},\\
  &\Lambda :=\sup_{M\times [0,T]}\hat A_f^+=\max \left\{0,\sup_{M\times [0,T]}(A'(u)-A(u)/u)\right\},\\
  &\Sigma :=\sup_{M\times [0,T]}\hat A_{ff}^+=\max \left\{0,\sup_{M\times [0,T]}(uA''(u)-A'(u)+A(u)/u)\right\}\\
  &\kappa :=-\min \left\{0,\inf_{Q_{2R,T}}\left(A'(u)-A(u)/u\right),\inf_{Q_{2R,T}}A'(u)\right\}.
\end{align*}
Here, we denote by $v^+=\max \{0,v\}$ and $v^-=\min \{0,v\}$ the positive part and the negative part of a function $v$. Notice that if $M$ is compact, then $\lambda ,\Lambda ,\Sigma $ and $\kappa $ must be finite.

\section{space-time gradient estimates for positive solutions}\label{sec2}

Firstly, we have the following local space-time gradient estimate for \eqref{eqa}--\eqref{eqb}.
\begin{thm}\label{tha}
  Let $(M,g(0))$ be a complete Riemannian manifold, and let $g(t)$ evolves by \eqref{eqb} for $t\in [0,T]$. Given $x_0$ and $R>0$, let $u$ be a positive solution to \eqref{eqa} in the cube $Q_{2R,T}:=\{(x,t):d(x,x_0,t)\le 2R, 0\le t\le T\}$. Suppose that there exist constants $K_1, K_2, K_3, K_4,\gamma ,\theta \ge 0$ such that 
  \[
  \Ric \ge -K_1g,\quad -K_2g\le h\le K_3g,\quad |\nabla h|\le K_4
  \]
  and
  \[
  |\nabla q|\le \gamma_{2R},\quad \Delta q\le \theta_{2R} 
  \]
  on $Q_{2R,T}$. Then for any $\alpha >1$ and $0<\varepsilon <1$, we have 
  \begin{equation}\label{eql}
    \begin{split}
      &\frac{|\nabla u(x,t)|^2}{u^2(x,t)}-\alpha \frac{u_t(x,t)}{u(x,t)}-\alpha q(x,t)-\alpha \frac{A(u(x,t))}{u(x,t)}\\
      \le &\frac{n\alpha^2}{t}+\frac{C\alpha^2}{R^2}\left(\frac{\alpha^2}{\alpha -1}+\sqrt{K_1}R\right)+C\alpha^2K_2+n\alpha^2\lambda_{2R}\\
      &+\Bigg\{n\alpha^2\bigg[\alpha \theta_{2R}+n\alpha^2\max \{K_2^2,K_3^2\}\\
      &\ \qquad +\frac{n\alpha^2}{4(1-\varepsilon )(\alpha -1)^2}\Big((\alpha -1)\Lambda_{2R}+\alpha \Sigma_{2R}+2(K_1+(\alpha -1)K_3+K_4)\Big)^2\\
      &\ \qquad +\frac{9}{8}n\alpha^2K_4+\frac{3}{4}\left(\frac{4n\alpha^2}{\varepsilon }\right)^{\frac{1}{3}}(\alpha -1)^{\frac{2}{3}}\gamma_{2R}^{\frac{4}{3}}\bigg]\Bigg\}^{\frac{1}{2}}
    \end{split}
  \end{equation}
  on $Q_{R,T}$, where $C$ is a constant that depends only on $n$. 
\end{thm}

\begin{rem}
  We see that Theorem~\ref{tha} covers \cite[Theorem~1]{sun2011Gradient}. In fact, when $q(x,t)=A(u)=0$, from Theorem~\ref{tha} we can get
  \begin{equation*}
    \begin{split}
      \frac{|\nabla u(x,t)|^2}{u^2(x,t)}-\alpha \frac{u_t(x,t)}{u(x,t)}\le &\frac{n\alpha^2}{t}+\frac{C\alpha^2}{R^2}\left(\frac{\alpha^2}{\alpha -1}+\sqrt{K_1}R\right)+C\alpha^2K_2\\
      &+\Bigg\{n\alpha^2\bigg[n\alpha^2\max \{K_2^2,K_3^2\}+\frac{9}{8}n\alpha^2K_4\\
      &+\frac{n\alpha^2}{(1-\varepsilon )(\alpha -1)^2}(K_1+(\alpha -1)K_3+K_4)^2\bigg]\Bigg\}^{\frac{1}{2}}.
    \end{split}
  \end{equation*} 
  Let $\varepsilon \to 0+$, we thus get
  \begin{equation*}
    \begin{split}
      \frac{|\nabla u(x,t)|^2}{u^2(x,t)}-\alpha \frac{u_t(x,t)}{u(x,t)}\le &\frac{n\alpha^2}{t}+\frac{C\alpha^2}{R^2}\left(R\sqrt{K_1}+\frac{\alpha^2}{\alpha -1}\right)+C\alpha^2K_2\\
      &+\frac{n\alpha^2}{\alpha -1}(K_1+(\alpha -1)K_3+K_4)\\
      &+n\alpha^2\left(\max \{K_2,K_3\}+\sqrt{9K_4/8}\right)^2.
    \end{split}
  \end{equation*} 
\end{rem}

To prove Theorem~\ref{tha}, we need the following two lemmas. Let $f=\log u$, by \eqref{eqa} we know that $f$ satisfies 
\begin{equation}\label{eqc}
  \begin{split}
    \Delta f=f_t+q+\frac{A(u)}{u}-|\nabla f|^2=f_t+q+\hat A(f)-|\nabla f|^2
  \end{split}
\end{equation}
Set $F=t(|\nabla f|^2-\alpha f_t-\alpha q-\alpha \hat A)$. We have
\begin{lem}[Lemma~3 in \cite{sun2011Gradient}]\label{lea}
  Suppose the metric evolves by \eqref{eqb}. Then for any smooth function $f$, we have
  \[
  \frac{\partial }{\partial t}|\nabla f|^2=-2h(\nabla f,\nabla f)+2\langle \nabla f,\nabla (f_t)\rangle 
  \]
  and 
  \begin{equation*}
    \begin{split}
      (\Delta f)_t=\Delta(f_t)-2\langle h,\Hess f\rangle -2\langle \Div h-\tfrac{1}{2}\nabla ({\tr}_gh),\nabla f\rangle ,
    \end{split}
  \end{equation*}
  where $\Div h$ is the divergence of $h$.
\end{lem}

\begin{lem}\label{leb}
  Let $(M,g(t))$ satisfies the hypotheses of Theorem~\ref{tha}. Then for any $\delta \in(0,\frac{1}{\alpha })$, we have
  \begin{equation}\label{eqd}
    \begin{split}
      (\Delta -\partial_t)F
      \ge &\frac{2(1-\delta \alpha)t}{n}(|\nabla f|^2-f_t-q-\hat A)^2-\frac{F}{t}-2\langle \nabla f, \nabla F\rangle \\
      &+\alpha t\hat A_f(|\nabla f|^2-f_t-q-\hat A)-2(\alpha -1)t\langle \nabla f,\nabla q\rangle \\
      &-t(2(\alpha -1)\hat A_f+\alpha \hat A_{ff}+2K_1+2(\alpha -1)K_3)|\nabla f|^2\\
      &-3\alpha t\sqrt{n}K_4|\nabla f|-\frac{\alpha tn}{2\delta }\max \{K_2^2,K_3^2\}-\alpha t\Delta q.
    \end{split}
  \end{equation}
\end{lem}
\begin{proof}
  By the Bochner formula, \eqref{eqc} and Lemma~\ref{lea}, we calculate 
  \begin{equation*}
    \begin{split}
      \Delta F=&2t|\Hess f|^2+2t\Ric (\nabla f,\nabla f)+2t\langle \nabla f,\nabla \Delta f\rangle \\
      &-\alpha t\Delta (f_t)-\alpha t\Delta q-\alpha t\hat A_f\Delta f-\alpha t\hat A_{ff}|\nabla f|^2\\
      =&2t|\Hess f|^2+2t\Ric (\nabla f,\nabla f)+2t\langle \nabla f,\nabla \Delta f\rangle \\
      &-\alpha t(\Delta f)_t-2\alpha t\langle h,\Hess f\rangle -2\alpha t\langle \Div h-\tfrac{1}{2}\nabla ({\tr}_gh),\nabla f\rangle \\
      &-\alpha t\Delta q-\alpha ta\hat A_f\Delta f-\alpha ta\hat A_{ff}|\nabla f|^2.
    \end{split}
  \end{equation*}
  By \eqref{eqc} and the definition of $F$ we have 
  \[
  \nabla \Delta f=-\frac{\nabla F}{t}-(\alpha -1)(\nabla (f_t)+\nabla q+\hat A_f\nabla f)
  \]
  and
  \[
  (\Delta f)_t=\frac{F}{t^2}-\frac{F_t}{t}-(\alpha -1)(f_{tt}+q_t+\hat A_ff_t).
  \]
  By Lemma~\ref{lea} we also have 
  \begin{equation*}
    \begin{split}
      F_t=&|\nabla f|^2-\alpha f_t-\alpha q-\alpha \hat A\\
      &+2t\langle \nabla f,(\nabla f)_t\rangle -\alpha t(f_{tt}+q_t+\hat A_ff_t)\\
      =&|\nabla f|^2-\alpha f_t-\alpha q-\alpha \hat A\\
      &+2t\langle \nabla f,\nabla (f_t)\rangle -2th(\nabla f,\nabla f)-\alpha t(f_{tt}+q_t+\hat A_ff_t).
    \end{split}
  \end{equation*}
  It follows the above equalities that
  \begin{equation}\label{eqe}
    \begin{split}
      (\Delta -\partial_t)F=&2t|\Hess f|^2+2t\Ric (\nabla f,\nabla f)-\frac{F}{t}-2\langle \nabla f,\nabla F\rangle \\
      &-2\alpha t\langle h,\Hess f\rangle -2\alpha t\langle \Div h-\tfrac{1}{2}\nabla ({\tr}_gh),\nabla f\rangle \\
      &+\alpha t\hat A_f(|\nabla f|^2-f_t-q-\hat A)-\alpha t\Delta q-2(\alpha -1)t\langle \nabla f,\nabla q\rangle \\
      &-2(\alpha -1)th(\nabla f,\nabla f)-2(\alpha -1)t\hat A_f|\nabla f|^2-\alpha t\hat A_{ff}|\nabla f|^2
    \end{split}
  \end{equation}
  The assumption $-K_2g\le h\le K_3g$ implies
  \[
  |h|^2\le n\max \{K_2^2,K_3^2\}.
  \]
  By Young's inequality, 
  \begin{equation}\label{eqf}
    \begin{split}
      \langle h,\Hess f\rangle \le &\delta |\Hess f|^2+\frac{1}{4\delta }|h|^2\\
      \le &\delta |\Hess f|^2+\frac{n}{4\delta }\max \{K_2^2,K_3^2\}
    \end{split}
  \end{equation}
  for any $\delta \in (0,\frac{1}{\alpha})$. We also have
  \begin{equation}\label{eqg}
    \begin{split}
      |\Div h-\tfrac{1}{2}\nabla ({\tr}_gh)|=&|g^{ij}\nabla_ih_{jl}-\tfrac{1}{2}g^{ij}\nabla_lh_{ij}|\\
      \le &\frac{3}{2}|g||\nabla h|\le \frac{3}{2}\sqrt{n}K_4.
    \end{split}
  \end{equation}
  On the other hand,
  \begin{equation}\label{eqh}
    \begin{split}
      |\Hess f|^2\ge \frac{1}{n}(\Delta f)^2=\frac{1}{n}(|\nabla f|^2-f_t-q-\hat A)^2.
    \end{split}
  \end{equation}
  Substituting \eqref{eqf}, \eqref{eqg} and \eqref{eqh} into \eqref{eqe} and using the assumptions on bounds of $\Ric $ and $h$, we obtain the final inequality \eqref{eqd}.
\end{proof}

\begin{proof}[The proof of Theorem~\ref{tha}]
  By the assumption of bounds of Ricci tensor and the evolution of the metric, we know that $g(t)$ is uniformly equivalent to the initial metric $g(0)$ (see \cite[Corollary~6.11]{chow2006Hamilton}), that is,
  \[
  e^{-2K_2T}g(0)\le g(t)\le e^{2K_3T}g(0).
  \]
  Then we know that $(M,g(t))$ is also complete for $t\in [0,T]$.

  Let $\phi \in C^\infty ((0,+\infty ])$, 
  \begin{equation*}
    \phi (s)=
    \begin{cases}
      1, & s\in [0,1],\\
      0, & s\in [2,+\infty )
    \end{cases}
  \end{equation*}
  satisfies $\phi(s)\in [0,1], \phi'(s)\le 0, \phi''(s)\ge -C_1$ and $\frac{|\phi'(s)|^2}{\phi(s)}\le C_1$, where $C_1$ is an absolute constant. Define
  \[
  \eta(x,t)=\phi\left(\frac{r(x,t)}{R}\right),
  \]
  where $r(x,t)=d(x,x_0,t)$. Using the argument of \cite{calabi1958An}, we can assume that the function $\eta(x,t)$ is $C^2$ with support in $Q_{2R,T}$.

  Define $G=\eta F$. For any $T_1\in (0,T]$, let $(x_1,t_1)\in Q_{2R,T_1}$ at which $G$ attains its maximum, and without loss of generality, we can assume $G(x_1,t_1)>0$, and then $\eta(x_1,t_1)>0$ and $F(x_1,t_1)>0$. Hence, at $(x_1,t_1)$, we have 
  \[
  \nabla G=0,\quad \Delta G\le 0,\quad \partial_tG\ge 0.
  \]
  Hence, we obtain 
  \begin{equation}\label{eqi}
    \begin{split}
      \nabla F=-\frac{F}{\eta }\nabla \eta 
    \end{split}
  \end{equation}
  and 
  \begin{equation}\label{eqj}
    \begin{split}
      0\ge &(\Delta -\partial_t)G\\
      =&F(\Delta -\partial_t)\eta +\eta (\Delta -\partial_t)F+2\langle \nabla \eta ,\nabla F\rangle .
    \end{split}
  \end{equation}
  By the properties of $\phi $ and the Laplacian comparison theorem, we have 
  \[
  \frac{|\nabla \eta |^2}{\eta }\le \frac{C_1}{R^2}
  \]
  and
  \begin{equation*}
    \begin{split}
      \Delta \eta =&\phi'\frac{\Delta r}{R}+\phi''\frac{|\nabla r|^2}{R^2}\\
      \ge &-\frac{\sqrt{C_1}}{R}\sqrt{(n-1)K_1}\coth \left(\sqrt{\tfrac{K_1}{n-1}}R\right)-\frac{C_1}{R^2}\\
      \ge &-\frac{(n-1)\sqrt{C_1}}{R^2}-\frac{\sqrt{(n-1)C_1K_1}}{R}-\frac{C_1}{R^2}.
    \end{split}
  \end{equation*}
  By \cite[p. 494]{sun2011Gradient}, there exist a constant $C_2$ such that
  \[
  -F\eta_t\ge -C_2K_2F.
  \]
  Substituting the above three inequalities into \eqref{eqj} and using \eqref{eqi}, we obtain 
  \begin{equation}\label{eqk}
    \begin{split}
      0\ge &-\left(\frac{(n-1)\sqrt{C_1}}{R^2}+\frac{\sqrt{(n-1)C_1K_1}}{R}+\frac{3C_1}{R^2}+C_2K_2\right)F+\eta (\Delta -\partial_t)F\\
      =:&\eta (\Delta -\partial_t)F-\left(\frac{C_3(n)}{R^2}+\frac{C_4(n)}{R}\sqrt{K_1}+C_2K_2\right)F.
    \end{split}
  \end{equation}
  
  Let $\mu =\frac{|\nabla f(x_1,t_1)|^2}{F(x_1,t_1)}$, then, at $(x_1,t_1)$ we have 
  \[
  \eta \langle \nabla f,\nabla F\rangle =-F\langle \nabla f,\nabla \eta \rangle \le \frac{\sqrt{C_1}}{R}\eta^{\frac{1}{2}}F|\nabla f| 
  \]
  and 
  \[
  |\nabla f|^2-f_t-q-a\hat A=\left(\mu -\frac{t_1\mu -1}{t_1\alpha }\right)F.
  \]
  Therefore, at $(x_1,t_1)$, by Lemma~\ref{leb} and \eqref{eqk}, and using the inequality
  \begin{equation}\label{eqo}
    \begin{split}
      3\alpha \sqrt{n}K_4|\nabla f|\le 2K_4|\nabla f|^2+\frac{9}{8}n\alpha^2K_4,
    \end{split}
  \end{equation}
  we obtain
  \begin{equation*}
    \begin{split}
      0\ge &\frac{2(1-\delta \alpha )t_1}{n}\eta \left(\mu -\frac{t_1\mu -1}{t_1\alpha }\right)^2F^2-\frac{\eta F}{t_1}-\frac{2\sqrt{C_1}}{R}\eta^{\frac{1}{2}}\mu^{\frac{1}{2}}F^{\frac{3}{2}}\\
      &+\alpha t_1\eta \hat A_f\left(\mu -\frac{t_1\mu -1}{t_1\alpha }\right)F-\alpha t_1\eta \Delta q-2(\alpha -1)t_1\eta \langle \nabla f,\nabla q\rangle \\
      &-t_1\eta (2(\alpha -1)\hat A_f+\alpha \hat A_{ff}+2(K_1+(\alpha -1)K_3+K_4))\mu F\\
      &-\frac{9}{8}t_1\eta n\alpha^2K_4-\frac{t_1\eta n\alpha }{2\delta }\max \{K_2^2,K_3^2\}\\
      &-\left(\tfrac{C_3(n)}{R^2}+\tfrac{C_4(n)}{R}\sqrt{K_1}+C_2K_2\right)F.
    \end{split}
  \end{equation*}
  Multiplying through by $t_1\eta $, we conclude that
  \begin{equation*}
    \begin{split}
      0\ge &\frac{2(1-\delta \alpha )}{n\alpha^2}(1+(\alpha -1)t_1\mu )^2G^2-\eta G-\frac{2\sqrt{C_1}}{R}t_1\mu^{\frac{1}{2}}G^{\frac{3}{2}}\\
      &+t_1\eta \hat A_f(1+(\alpha -1)t_1\mu )G-\alpha t_1^2\eta^2\Delta q-2(\alpha -1)t_1^2\eta^2\langle \nabla f,\nabla q\rangle \\
      &-t_1^2\eta[2(\alpha -1)\hat A_f+\alpha \hat A_{ff}+2(K_1+(\alpha -1)K_3+K_4)]\mu G\\
      &-\frac{9}{8}t_1^2\eta^2n\alpha^2K_4-\frac{t_1^2\eta^2n\alpha }{2\delta }\max \{K_2^2,K_3^2\}\\
      &-t_1\left(\tfrac{C_3(n)}{R^2}+\tfrac{C_4(n)}{R}\sqrt{K_1}+C_2K_2\right)G\\
      =&\frac{2(1-\delta \alpha )}{n\alpha^2}(1+(\alpha -1)t_1\mu )^2G^2-\eta G-\frac{2\sqrt{C_1}}{R}t_1\mu^{\frac{1}{2}}G^{\frac{3}{2}}\\
      &+t_1\eta \hat A_fG-\alpha t_1^2\eta^2\Delta q-2(\alpha -1)t_1^2\eta^2\langle \nabla f,\nabla q\rangle \\
      &-t_1^2\eta[(\alpha -1)\hat A_f+\alpha \hat A_{ff}+2(K_1+(\alpha -1)K_3+K_4)]\mu G\\
      &-\frac{9}{8}t_1^2\eta^2n\alpha^2K_4-\frac{t_1^2\eta^2n\alpha }{2\delta }\max \{K_2^2,K_3^2\}\\
      &-t_1\left(\tfrac{C_3(n)}{R^2}+\tfrac{C_4(n)}{R}\sqrt{K_1}+C_2K_2\right)G
    \end{split}
  \end{equation*}
  Noticing that $0<\eta(x_1,t_1) \le 1$, from the above inequalities we obtain
  \begin{equation*}
    \begin{split}
      0\ge &\frac{2(1-\delta \alpha )}{n\alpha^2}G^2+\frac{4(1-\delta \alpha )}{n\alpha^2}(\alpha -1)t_1\mu G^2\\
      &+\frac{2(1-\delta \alpha )}{n\alpha^2}(\alpha -1)^2t_1^2\mu^2G^2-G-t_1\lambda_{2R}G-\frac{2\sqrt{C_1}}{R}t_1\mu^{\frac{1}{2}}G^{\frac{3}{2}}\\
      &-t_1^2\Big[(\alpha -1)\Lambda_{2R}+\alpha \Sigma_{2R}+2(K_1+(\alpha -1)K_3+K_4)\Big]\mu G\\
      &-2t_1^2(\alpha -1)\gamma_{2R}\mu^{\frac{1}{2}}G^{\frac{1}{2}}-\frac{9}{8}t_1^2n\alpha^2K_4-\frac{t_1^2n\alpha }{2\delta }\max \{K_2^2,K_3^2\}\\
      &-\alpha t_1^2\theta_{2R}-t_1\left(\tfrac{C_3(n)}{R^2}+\tfrac{C_4(n)}{R}\sqrt{K_1}+C_2K_2\right)G.
    \end{split}
  \end{equation*}
  By Young's inequality, we have 
  \[
  \frac{2\sqrt{C_1}}{R}\mu^{\frac{1}{2}}G^{\frac{3}{2}}\le \frac{4(1-\delta \alpha )}{n\alpha^2}(\alpha -1)\mu G^2+\frac{n\alpha^2C_1G}{8(1-\delta \alpha )(\alpha -1)R^2},
  \]
  \begin{equation*}
    \begin{split}
      2(\alpha -1)\gamma_{2R}\mu^{\frac{1}{2}}G^{\frac{1}{2}}\le \frac{2(1-\delta \alpha )\varepsilon }{n\alpha^2}(\alpha -1)^2\mu^2G^2+\frac{3}{4}\left(\frac{2n\alpha^2}{(1-\delta \alpha )\varepsilon }\right)^{\frac{1}{3}}(\alpha -1)^{\frac{2}{3}}\gamma_{2R}^{\frac{4}{3}}
    \end{split}
  \end{equation*}
  and 
  \begin{equation*}
    \begin{split}
      &\Big[(\alpha -1)\Lambda_{2R}+\alpha \Sigma_{2R}+2(K_1+(\alpha -1)K_3+K_4)\Big]\mu G\\
      \le &\frac{2(1-\delta \alpha )(1-\varepsilon )}{n\alpha^2}(\alpha -1)^2\mu^2G^2\\
      &+\frac{n\alpha^2}{8(1-\delta \alpha )(1-\varepsilon )(\alpha -1)^2}\Big[(\alpha -1)\Lambda_{2R}+\alpha \Sigma_{2R}+2(K_1+(\alpha -1)K_3+K_4)\Big]^2,
    \end{split}
  \end{equation*}
  where $\varepsilon \in (0,1)$ is an arbitrary constant. Combining the above four inequalities, there exists a constant $C_5(n)$ that depends only on $n$, such that
  \begin{equation*}
    \begin{split}
      0\ge &\frac{2(1-\delta \alpha )}{n\alpha^2}G^2\\
      &-\left[1+t_1\left(\lambda_{2R}+\frac{C_5(n)}{R^2}\left(\frac{\alpha^2}{(1-\delta \alpha )(\alpha -1)}+\sqrt{K_1}R\right)+C_2K_2\right)\right]G\\
      &-t_1^2\alpha \theta_{2R}-t_1^2n\alpha^2\max \{K_2^2,K_3^2\}\\
      &-t_1^2\frac{n\alpha^2}{8(1-\delta \alpha )(1-\varepsilon )(\alpha -1)^2}\\
      &\cdot \Big[(\alpha -1)\Lambda_{2R}+\alpha \Sigma_{2R}+2(K_1+(\alpha -1)K_3+K_4)\Big]^2\\
      &-\frac{9}{8}t_1^2n\alpha^2K_4-\frac{3}{4}t_1^2\left(\frac{2n\alpha^2}{(1-\delta \alpha )\varepsilon }\right)^{\frac{1}{3}}(\alpha -1)^{\frac{2}{3}}\gamma_{2R}^{\frac{4}{3}}.
    \end{split}
  \end{equation*}
  For a positive number $a$ and two nonnegative numbers $b,c$, from the inequality $ax^2-bx-c\le 0$ we have $x\le \frac{b}{a}+\sqrt{\frac{c}{a}}$. Hence, we obtain
  \begin{equation*}
    \begin{split}
      G\le &\frac{n\alpha^2}{2(1-\delta \alpha )}\left[1+t_1\left(\lambda_{2R}+\frac{C_5(n)}{R^2}\left(\frac{\alpha^2}{(1-\delta \alpha )(\alpha -1)}+\sqrt{K_1}R\right)+C_2K_2\right)\right]\\
      &+t_1\Bigg\{n\alpha^2\bigg[\alpha \theta_{2R}+n\alpha^2\max \{K_2^2,K_3^2\}\\
      &\ \qquad +\frac{n\alpha^2}{8(1-\delta \alpha )(1-\varepsilon )(\alpha -1)^2}\\
      &\quad \quad \ \cdot \Big[(\alpha -1)\Lambda_{2R}+\alpha \Sigma_{2R}+2(K_1+(\alpha -1)K_3+K_4)\Big]^2\\
      &\ \qquad +\frac{9}{8}n\alpha^2K_4+\frac{3}{4}\left(\frac{2n\alpha^2}{(1-\delta \alpha )\varepsilon }\right)^{\frac{1}{3}}(\alpha -1)^{\frac{2}{3}}\gamma_{2R}^{\frac{4}{3}}\bigg]\Bigg\}^{\frac{1}{2}}. 
    \end{split}
  \end{equation*}
  Now, by taking $\delta =\frac{1}{2\alpha }$, and noticing that $d(x,x_0,T_1)\le R$ implies $\eta(x,T_1)=1$, we can get
  \begin{equation*}
    \begin{split}
      &(|\nabla f|^2-\alpha f_1-\alpha q_t-\alpha \hat A)(x,T_1)=\frac{F(x,T_1)}{T_1}\le \frac{G(x_1,t_1)}{T_1}\\
      \le &\frac{n\alpha^2}{T_1}+\frac{C(n)\alpha^2}{R^2}\left(\frac{\alpha^2}{\alpha -1}+\sqrt{K_1}R\right)+C(n)\alpha^2K_2+n\alpha^2\lambda_{2R}\\
      &+\Bigg\{n\alpha^2\bigg[\alpha \theta_{2R}+n\alpha^2\max \{K_2^2,K_3^2\}\\
      &\ \qquad +\frac{n\alpha^2}{4(1-\varepsilon )(\alpha -1)^2}\Big((\alpha -1)\Lambda_{2R}+\alpha \Sigma_{2R}+2(K_1+(\alpha -1)K_3+K_4)\Big)^2\\
      &\ \qquad +\frac{9}{8}n\alpha^2K_4+\frac{3}{4}\left(\frac{4n\alpha^2}{\varepsilon }\right)^{\frac{1}{3}}(\alpha -1)^{\frac{2}{3}}\gamma_{2R}^{\frac{4}{3}}\bigg]\Bigg\}^{\frac{1}{2}},
    \end{split}
  \end{equation*}
  where $C(n)$ is an appropriate constant that depends only $n$. Since $T_1$ is arbitrary, we complete the proof.
\end{proof}

\begin{rem}
  In the above proof, if we use $x\le \frac{1}{2a}\left(b+\sqrt{b^2+4ac}\right)$ instead of $x\le \frac{b}{a}+\sqrt{\frac{c}{a}}$ when we deal with $ax^2-bx-c\le 0$, then a more appropriate $\delta $ may give a sharper estimate.
\end{rem}

From the above local estimate, we get a global one:
\begin{cor}\label{coa}
  Let $(M,g(0))$ be a complete noncompact Riemannian manifold without boundary, and let $g(t)$ evolves by \eqref{eqb} for $t\in [0,T]$. Suppose that there exist constant $K_1, K_2, K_3, K_4, \gamma , \theta \ge 0$ such that 
  \[
  \Ric \ge -K_1g,\quad -K_2g\le h\le K_3g,\quad |\nabla h|\le K_4
  \]
  and 
  \[
  |\nabla q|\le \gamma ,\quad \Delta q\le \theta .
  \]
  If $u$ is a positive solution to \eqref{eqa}, then for any $\alpha >1$, we have 
  \begin{equation}\label{eqm}
    \begin{split}
      &\frac{|\nabla u(x,t)|^2}{u^2(x,t)}-\alpha \frac{u_t(x,t)}{u(x,t)}-\alpha q(x,t)-\alpha \frac{A(u(x,t))}{u(x,t)}\\
      \le &\frac{n\alpha^2}{t}+C'\left(K_1+K_2+K_3+K_4+\sqrt{K_4}+\sqrt{\theta }+\gamma^{\frac{2}{3}}+\lambda +\Lambda +\Sigma \right)
    \end{split}
  \end{equation}
  on $M\times [0,T]$, where $C'$ is a constant that depends only on $n,\alpha $.
\end{cor}
\begin{proof}
  By the uniform equivalence of $g(t)$, we know that $(M,g(t))$ is complete noncompact for $t\in [0,T]$. Now we choose $\varepsilon =\varepsilon_0$ in \eqref{eql}, where $\varepsilon_0$ is an arbitrary fixed number in $(0,1)$. Let $R\to +\infty $ in \eqref{eql}, and using the inequality $\sqrt{x+y}\le \sqrt{x}+\sqrt{y}$ holds for any $x,y\ge 0$, we complete the proof.
\end{proof} 

We now consider the case that the manifold $M$ is closed. By Lemma~\ref{leb}, we have a global gradient estimate on a closed Riemannian manifold.

\begin{thm}\label{thb}
  Let $(M,g(t))$ be a closed Riemannian manifold, where $g(t)$ evolves by \eqref{eqb} for $t\in [0,T]$ and satisfies
  \[
  \Ric \ge -K_1g,\quad -K_2g\le h\le K_3g,\quad |\nabla h|\le K_4.
  \]
  If $u$ is a positive solution to \eqref{eqa}, and $q(x,t)$ satisfies
  \[
  |\nabla q|\le \gamma ,\quad \Delta q\le \theta .
  \]
  Then for any $\alpha >1$, we have 
  \begin{equation}\label{eqn}
    \begin{split}
      &\frac{|\nabla u(x,t)|^2}{u^2(x,t)}-\alpha \frac{u_t(x,t)}{u(x,t)}-\alpha q(x,t)-\alpha \frac{A(u(x,t))}{u(x,t)}\\
      \le &\frac{n\alpha^2}{2t}+\frac{n\alpha^2}{\alpha -1}(K_1+(\alpha -1)K_3+K_4)+n\alpha^2\left(\max \{K_2,K_3\}+\frac{3}{4}\sqrt{2K_4}\right)\\
      &+\alpha^{\frac{3}{2}}\sqrt{n\theta }+\left(\frac{n\alpha^2}{2}(\alpha -1)^{\frac{1}{2}}+\sqrt{n}\alpha (\alpha -1)^{\frac{1}{4}}\right)\gamma^{\frac{2}{3}}+\frac{n\alpha^2}{2}(\lambda +\Lambda )+\frac{n\alpha^3}{2(\alpha -1)}\Sigma 
    \end{split}
  \end{equation}
  on $M\times (0,T]$.
\end{thm}
\begin{proof}
  We use the same symbols $f, F$ as above. Set 
  \begin{equation*}
    \begin{split}
      \bar F(x,t)=&F(x,t)-\frac{n\alpha^2}{\alpha -1}(K_1+(\alpha -1)K_3+K_4)t\\
      &-n\alpha^2\left(\max \{K_2,K_3\}+\frac{3}{4}\sqrt{2K_4}\right)t-\alpha^{\frac{3}{2}}\sqrt{n\theta }t\\
      &-\left(\frac{n\alpha^2}{2}(\alpha -1)^{\frac{1}{2}}+\sqrt{n}\alpha (\alpha -1)^{\frac{1}{4}}\right)\gamma^{\frac{2}{3}}t-\frac{n\alpha^2}{2}(\lambda +\Lambda )t-\frac{n\alpha^3}{2(\alpha -1)}\Sigma t.
    \end{split}
  \end{equation*}
  
  If $\bar F(x,t)\le \frac{n\alpha^2}{2}$ for any $(x,t)\in M\times (0,T]$, the proof is complete.

  If \eqref{eqn} doesn't hold, then at the maximal point $(x_0,t_0)$ of $\bar F(x,t)$, we have 
  \[
  \bar F(x_0,t_0)>\frac{n\alpha^2}{2}.
  \]
  As $\bar F(x,0)=0$, we know that $t_0>0$ here. Then applying the maximum principle, we have
  \[
  \nabla \bar F(x_0,t_0)=0,\quad \Delta \bar F(x_0,t_0)=0,\quad \partial_t\bar F(x_0,t_0)=0.
  \]
  Therefore, we obtain 
  \[
  0\ge (\Delta -\partial_t)\bar F\ge (\Delta -\partial_t)F.
  \] 
  Using Lemma~\ref{leb}, inequality \eqref{eqo} and the fact that 
  \begin{equation*}
    \begin{split}
      |\nabla f|^2-f_t-q-\hat A=&\frac{1}{\alpha }(|\nabla f|^2-\alpha f_t-\alpha q-\alpha \hat A)+\frac{\alpha -1}{\alpha }|\nabla f|^2\\
      =&\frac{1}{\alpha }\frac{F}{t_0}+\frac{\alpha -1}{\alpha }|\nabla f|^2,
    \end{split}
  \end{equation*}
  we obtain
  \begin{equation*}
    \begin{split}
      0\ge &\frac{2(1-\delta \alpha )t_0}{n\alpha^2}\left(\frac{F}{t_0}\right)^2+\frac{4(1-\delta \alpha )(\alpha -1)t_0}{n\alpha^2}|\nabla f|^2\frac{F}{t_0}\\
      &+\frac{2(1-\delta \alpha )(\alpha -1)^2t_0}{n\alpha^2}|\nabla f|^4-\frac{F}{t_0}-t_0\lambda \frac{F}{t_0}\\
      &-2t_0(\alpha -1)\gamma |\nabla f|-\frac{9}{8}t_0n\alpha^2K_4-t_0\alpha \theta -\frac{t_0n\alpha }{2\delta }\max \{K_2^2,K_3^2\}\\
      &-t_0\left((\alpha -1)\Lambda +\alpha \Sigma +2(K_1+(\alpha -1)K_3+K_4)\right)|\nabla f|^2. 
    \end{split}
  \end{equation*}
  By 
  \begin{equation*}
    \begin{split}
       \frac{F}{t_0}=&\frac{\bar F}{t_0}+\frac{n\alpha^2}{\alpha -1}(K_1+(\alpha -1)K_3+K_4)+\frac{3}{2}n\alpha^2\sqrt{K_4}\\
       &+\left(\frac{n\alpha^2}{2}(\alpha -1)^{\frac{1}{2}}+\sqrt{2n}\alpha (\alpha -1)^{\frac{1}{4}}\right)\gamma^{\frac{2}{3}}\\
       &+n\alpha^2A_1+\frac{n\alpha^3}{2(\alpha -1)}A_2>0,
    \end{split}
  \end{equation*}
  \[
  2t_0(\alpha -1)\gamma |\nabla f|^2\le t_0(\alpha -1)^{\frac{3}{2}}\gamma^{\frac{2}{3}}|\nabla f|^2+t_0(\alpha -1)^{\frac{1}{2}}\gamma^{\frac{4}{3}},
  \]
  and using the inequality $ax^2-bx\ge -\frac{b^2}{4a}$ holds for $a>0, b\ge 0$, we obtain
  \begin{equation*}
    \begin{split}
      0\ge &\frac{2(1-\delta \alpha )t_0}{n\alpha^2}\left(\frac{F}{t_0}\right)^2-\frac{F}{t_0}-t_0\lambda \frac{F}{t_0}-\frac{9}{8}t_0n\alpha^2K_4-t_0\alpha \theta \\
      &-\frac{t_0n\alpha }{2\delta }\max \{K_2^2,K_3^2\}-t_0(\alpha -1)^{\frac{1}{2}}\gamma^{\frac{4}{3}}-\frac{t_0n\alpha^2}{8(1-\delta \alpha )(\alpha -1)^2}E^2,
    \end{split}
  \end{equation*}
  where 
  \[
  E=(\alpha -1)^{\frac{3}{2}}\gamma^{\frac{2}{3}}+(\alpha -1)\Lambda +\alpha \Sigma +2(K_1+(\alpha -1)K_3+K_4).
  \]
  For a positive number $a$ and two nonnegative numbers $b,c$, from the inequality $ax^2-bx-c\le 0$ we have 
  \[
  x\le \frac{1}{2a}\left(b+\sqrt{b^2+4ac}\right).
  \] 
  Hence, we obtain
  \begin{equation*}
    \begin{split}
      \frac{F}{t_0}\le &\frac{n\alpha^2}{4(1-\delta \alpha )t_0}\Bigg\{1+t_0\lambda +\bigg[(1+t_0\lambda )^2\\
      &+\frac{8(1-\delta \alpha )t_0}{n\alpha^2}\bigg(\frac{n\alpha^2t_0}{8(1-\delta \alpha )(\alpha -1)^2}E^2\\
      &+t_0(\alpha -1)^{\frac{1}{2}}\gamma^{\frac{4}{3}}+\frac{9}{8}t_0n\alpha^2K_4+t_0\alpha \theta +\frac{t_0n\alpha }{2\delta }\max \{K_2^2,K_3^2\}\bigg)\bigg]^{\frac{1}{2}}\Bigg\}.
    \end{split}
  \end{equation*}
  Using the inequality $\sqrt{x+y}\le \sqrt{x}+\sqrt{y}$ holds for any $x,y\ge 0$, we obtain
  \begin{equation*}
    \begin{split}
      \frac{F}{t_0}\le &\frac{n\alpha^2}{4(1-\delta \alpha )t_0}\Bigg\{1+t_0\lambda +\bigg[(1+t_0\lambda )^2+\frac{4(1-\delta \alpha )t_0^2}{\delta \alpha }\max \{K_2^2,K_3^2\}\bigg]^{\frac{1}{2}}\Bigg\}\\
      &+\frac{n\alpha^2}{4(1-\delta \alpha )t_0}\Bigg\{\frac{8(1-\delta \alpha )t_0^2}{n\alpha^2}\bigg[\frac{n\alpha^2}{8(1-\delta \alpha )(\alpha -1)^2}E^2\\
      &+(\alpha -1)^{\frac{1}{2}}\gamma^{\frac{4}{3}}+\frac{9}{8}n\alpha^2K_4+\alpha \theta \bigg]\Bigg\}^{\frac{1}{2}}.
    \end{split}
  \end{equation*}
  Now, by taking $\delta =\frac{t_0\max \{K_2,K_3\}}{1+t_0\lambda +2t_0\max \{K_2,K_3\}}\cdot \frac{1}{\alpha }\in (0,\frac{1}{\alpha })$ for $\max \{K_2,K_3\}\ne 0$. However, from Lemma~\ref{leb} we know that we can choose $\delta =0$ if $K_2=K_3=0$. Therefore, in any case, we can get 
  \begin{equation*}
    \begin{split}
      &\frac{n\alpha^2}{4(1-\delta \alpha )t_0}\Bigg\{1+t_0\lambda +\bigg[(1+t_0\lambda )^2+\frac{4(1-\delta \alpha )t_0^2}{\delta \alpha }\max \{K_2^2,K_3^2\}\bigg]^{\frac{1}{2}}\Bigg\}\\
      =&\frac{n\alpha^2}{2t_0}+\frac{n\alpha^2}{2}\lambda +n\alpha^2\max \{K_2,K_3\}
    \end{split}
  \end{equation*}
  and 
  \begin{equation*}
    \begin{split}
      &\frac{n\alpha^2}{4(1-\delta \alpha )t_0}\left(\frac{8(1-\delta \alpha )t_0^2}{n\alpha^2}\right)^{\frac{1}{2}}=\left(\frac{1+t_0\lambda +2t_0\max \{K_2,K_3\}}{2(1+t_0\lambda +t_0\max \{K_2,K_3\})}\right)^{\frac{1}{2}}\sqrt{n}\alpha \le \sqrt{n}\alpha .
    \end{split}
  \end{equation*}
  Therefore, again according to $\sqrt{x+y}\le \sqrt{x}+\sqrt{y}$, we obtain
  \begin{equation*}
    \begin{split}
      \frac{F}{t_0}\le &\frac{n\alpha^2}{2t_0}+\frac{n\alpha^2}{2}\lambda +n\alpha^2\max \{K_2,K_3\}\\
      &+\sqrt{n}\alpha \Bigg\{\frac{\sqrt{n}\alpha }{2\sqrt{2}(\alpha -1)}\sqrt{\frac{1+t_0\lambda +2t_0\max \{K_2,K_3\}}{1+t_0\lambda +t_0\max \{K_2,K_3\}}}E\\
      &+(\alpha -1)^{\frac{1}{4}}\gamma^{\frac{2}{3}}+\frac{3}{2\sqrt{2}}\sqrt{n}\alpha \sqrt{K_4}+\sqrt{\alpha \theta }\Bigg\}\\
      \le &\frac{n\alpha^2}{2t_0}+\frac{n\alpha^2}{2}\lambda +n\alpha^2\max \{K_2,K_3\}\\
      &+\frac{n\alpha^2}{2(\alpha -1)}E+\sqrt{n}\alpha (\alpha -1)^{\frac{1}{4}}\gamma^{\frac{2}{3}}+\frac{3}{4}n\alpha^2\sqrt{2K_4}+\alpha^{\frac{3}{2}}\sqrt{n\theta }.
    \end{split}
  \end{equation*}
  Substituting $E$ into the above inequality yields
  \begin{equation*}
    \begin{split}
      \frac{F}{t_0}\le &\frac{n\alpha^2}{2t_0}+\frac{n\alpha^2}{\alpha -1}(K_1+(\alpha -1)K_3+K_4)\\
      &+n\alpha^2\left(\max \{K_2,K_3\}+\frac{3}{4}\sqrt{2K_4}\right)+\alpha^{\frac{3}{2}}\sqrt{n\theta }\\
      &+\left(\frac{n\alpha^2}{2}(\alpha -1)^{\frac{1}{2}}+\sqrt{n}\alpha (\alpha -1)^{\frac{1}{4}}\right)\gamma^{\frac{2}{3}}+\frac{n\alpha^2}{2}(\lambda +\Lambda )+\frac{n\alpha^3}{2(\alpha -1)}\Sigma .
    \end{split}
  \end{equation*} 
  This implies that $\bar F(x_0,t_0)\le \frac{n\alpha^2}{2}$, in contradiction with our assumption. So \eqref{eqn} holds.
\end{proof}

\begin{rem}
  In \cite[Theorem~6]{sun2011Gradient}, The coefficient of $\frac{1}{t}$ in the right hand side of the gradient inequality is $n\alpha^2$. We see Theorem~\ref{thb} extends and improves Sun's estimate.
\end{rem}

\begin{rem}\label{rea}
  In Theorem~\ref{tha} if $K_1=K_4=\Sigma_{2R}=0$, we can let $\alpha \to 1$. Similarly, in Corollary~\ref{coa} and Theorem~\ref{thb}, if $K_1=K_4=\Sigma =0$, we can also let $\alpha \to 1$.
\end{rem}

Similar to \cite[Corollary~8]{sun2011Gradient}, integrating the gradient estimate in space-time as in \cite{li1986parabolic} or \cite{guenther2002The}, we can derive the following parabolic Harnack type inequality.
\begin{cor}\label{coc}
  Let $(M,g(0))$ be a complete noncompact Riemannian manifold without boundary or a closed Riemannian manifold. Assume that $g(t)$ evolves by \eqref{eqb} for $t\in [0,T]$ and satisfies 
  \[
  \Ric \ge -K_1g,\quad -K_2g\le h\le K_3g,\quad |\nabla h|\le K_4.
  \]
  If $u$ is a positive solution to \eqref{eqa}, and $q(x,t)$ satisfies 
  \[
  |\nabla q|\le \gamma ,\quad \Delta q\le \theta .
  \]
  Then for any $(x_1,t_1), (x_2,t_2)$ in $M\times (0,T]$ such that $t_1<t_2$, we have 
  \begin{equation}\label{eqt}
    \begin{split}
      u(x_1,t_1)\le u(x_2,t_2)\left(\frac{t_2}{t_1}\right)^{\tau n\alpha }\exp \left(\frac{\alpha Z }{4(t_2-t_1)}+C\frac{t_2-t_1}{\alpha }K\right),
    \end{split}
  \end{equation} 
  for any $\alpha >1$, where 
  \begin{equation*}
    \tau =
    \begin{cases}
      1, & \mbox{if}\ (M,g(0))\ \mbox{is complete noncompact without boundary},\\
      \frac{1}{2}, & \mbox{if}\ (M,g(0))\ \mbox{is closed,}
    \end{cases}
  \end{equation*}
  \[K=K_1+K_2+K_3+K_4+\sqrt{K_4}+\gamma +\sqrt{\theta }+\gamma^{\frac{2}{3}}+\lambda +\Lambda +\Sigma ,
  \]
  $C$ is a constant that depends only on $n,\alpha $, and 
  \[
  Z=\inf_{\zeta }\int_0^1|\zeta'(s)|_{\sigma(s)}^2ds
  \] 
  is the infimum over smooth curves $\zeta $ jointing $x_2$ and $x_1$ ($\zeta(0)=x_2$, $\zeta(1)=x_1$) of the averaged square velocity of $\zeta $ measured at time $\sigma(s)=(1-s)t_2+st_1$.  
\end{cor}
\begin{proof}
  The gradient estimate in Corollary~\ref{coa} and Theorem~\ref{thb} can both be written as 
  \[
  \frac{|\nabla u(x,t)|^2}{u^2(x,t)}-\alpha \frac{u_t(x,t)}{u(x,t)}\le \frac{\tau n\alpha^2}{t}+C(n,\alpha )K,
  \]
  for any $\alpha >1$. Take any curve $\zeta $ satisfying the assumption and define
  \[
  l(s)=\log u(\zeta(s),\sigma(s)).
  \]
  Then $l(0)=\log u(x_2,t_2)$ and $l(1)=\log u(x_1,t_1)$. A direct computation yields 
  \begin{equation*}
    \begin{split}
      \frac{dl(s)}{ds}=&(t_2-t_1)\left(\frac{\nabla u}{u}\frac{\zeta'(s)}{t_2-t_1}-\frac{u_t}{u}\right)\\
      \le &\frac{\alpha |\zeta'(s)|_\sigma^2}{4(t_2-t_1)}+\frac{t_2-t_1}{\alpha }\left(\frac{\tau n\alpha^2}{\sigma(s)}+CK\right).
    \end{split}
  \end{equation*} 
  Integrating this inequality over $\zeta(s)$, we have
  \begin{equation*}
    \begin{split}
      \log \frac{u(x_1,t_1)}{u(x_2,t_2)}=&\int_0^1\frac{dl(s)}{ds}ds\\
      \le &\int_0^1\frac{\alpha |\zeta'(s)|_\sigma^2}{4(t_2-t_1)}+C\frac{t_2-t_1}{\alpha }K+\tau n\alpha \log \frac{t_2}{t_1},
    \end{split}
  \end{equation*}
  which implies the corollary.
\end{proof}

\section{Elliptic type gradient estimates for bounded positive solutions}\label{sec3}

Now we establish elliptic type gradient estimates for \eqref{eqa}--\eqref{eqb}. Firstly we give the local version.
\begin{thm}\label{thc}
  Let $(M,g(t))$ be a complete solution to \eqref{eqb} for $t\in [0,T]$ and let $u$ be a positive solution to \eqref{eqa}. Suppose that there exist constants $L>0$ and $K\ge 0$, such that $u\le L$ and 
  \[
  \Ric \ge -K_1g,\quad h\ge -K_2g
  \]
  on $Q_{2R,T}$. Then we have 
  \begin{equation}\label{eqp}
    \begin{split}
      \frac{|\nabla u(x,t)|}{u(x,t)}\le \tilde C\left(\frac{1}{\sqrt{t}}+\frac{1}{R}+\sqrt{H}\right)\left(1+\log \frac{L}{u(x,t)}\right)
    \end{split}
  \end{equation}
  on $Q_{R,T}$, where $\tilde C$ is a constant that depends only on $n$ and
  \[
  H=K_1+K_2+\max_{Q_{2R,T}}(|q|-q)+\max_{Q_{2R,T}}\frac{|\nabla q|}{\sqrt{|q|}}+\kappa_{2R}.
  \]
\end{thm}

\begin{rem}
  In \cite[Theorem~1.9]{chen2018Li}, the authors gave a elliptic type gradient estimate for bounded positive solutions of \eqref{eqa} with a convection term on a complete manifold, where the metric does not depend on time. In the estimate of \cite[Theorem~1.9]{chen2018Li}, the upper bound induced by the term $A(u)$ is 
  \[
  -\min \left\{0,\min_{Q_{2R,T}}(A'(u)-A(u)/u)\right\}-\min \left\{0,\min_{Q_{2R,T}}(A(u)/u)\right\}
  \]
  instead of $\kappa_{2R}$ here. Compare with \cite[Theorem~1.9]{chen2018Li}, we see that our estimate \eqref{eqp} is sharper. In fact, in general, for real numbers $x,y$, a direct calculation yields
  \[
  \min \{0,x\}+\min \{0,y\}\le \min \{0,x+y\}.
  \]
  Similarly, for two functions $f,g$ on the same domain $D$, the following obvious fact holds:
  \[
  \min_Df+\min_Dg\le \min_D(f+g).
  \]
  By the above two inequalities we obtain
  \begin{equation*}
    \begin{split}
      &\min \left\{0,\min_{Q_{2R,T}}(A'(u)-A(u)/u)\right\}+\min \left\{0,\min_{Q_{2R,T}}(A(u)/u)\right\}\\
      \le &\min \left\{0,\min_{Q_{2R,T}}\left(A'(u)-A(u)/u\right)+\min_{Q_{2R,T}}\left(A(u)/u\right)\right\}\\
      \le &\min \left\{0,\min_{Q_{2R,T}}A'(u)\right\}.
    \end{split}
  \end{equation*}
  On the other hand, it is obvious that  
  \begin{equation*}
    \begin{split}
      &\min \left\{0,\min_{Q_{2R,T}}(A'(u)-A(u)/u)\right\}+\min \left\{0,\min_{Q_{2R,T}}(A(u)/u)\right\}\\
      \le &\min \left\{0,\min_{Q_{2R,T}}(A'(u)-A(u)/u)\right\}.
    \end{split}
  \end{equation*}
  In conclusion, 
  \begin{equation*}
    \begin{split}
      &\min \left\{0,\min_{Q_{2R,T}}(A'(u)-A(u)/u)\right\}+\min \left\{0,\min_{Q_{2R,T}}(A(u)/u)\right\}\\
      \le &\min \left\{0,\min_{Q_{2R,T}}(A'(u)-A(u)/u),\min_{Q_{2R,T}}A'(u)\right\}=-\kappa_{2R}.
    \end{split}
  \end{equation*}
  That is,
  \[
  \kappa_{2R}\le -\min \left\{0,\min_{Q_{2R,T}}(A'(u)-A(u)/u)\right\}-\min \left\{0,\min_{Q_{2R,T}}(A(u)/u)\right\}.
  \]
  And we will see in the proof of Theorem~\ref{thc} that this sharper estimate comes from a more careful treatment of the term $\hat A_f+\frac{\hat A(f)}{1-f}$. However, the treatment we give here is not necessarily optimal. It is possible that a sharper estimate will be applied to more equations.
\end{rem}

Now we are ready to prove Theorem~\ref{thc}. Noticing that if $0<u\le L$ is a solution to \eqref{eqa}, then $\tilde u=\frac{u}{L}$ is a solution to the equation 
\[
(\Delta -q(x,t)-\partial_t)u(x,t)=\frac{1}{L}A(L\tilde u(x,t))
\] 
and $0<\tilde u\le 1$. Hence, we can assume that $0<u\le 1$ in the proof of Theorem~\ref{thc}. Similar to the proof of Theorem~\ref{tha}, we need a auxiliary lemma. We still set $f=\log u\le 0$ and $\hat A(f)=\frac{A(u)}{u}$. In this case, we define $w=|\nabla \log(1-f)|^2$ and $F(x,t)=tw(x,t)$. 
\begin{lem}\label{lec}
  Let $(M,g(t))$ be a complete solution to \eqref{eqb} for $t\in [0,T]$ and let $u\in (0,1]$ be a solution to \eqref{eqa}. Suppose that there exists a constant $K\ge 0$, such that 
  \[
  \Ric +h\ge -Kg
  \]
  on $Q_{2R,T}$. Then we have 
  \begin{equation}\label{eqq}
    \begin{split}
      (\Delta -\partial_t)F\ge &\frac{2(1-f)}{t}F^2-2\left(K-\hat A_f-\frac{q}{1-f}-\frac{\hat A(f)}{1-f}\right)F\\
      &-\frac{F}{t}-2f\langle \nabla \log(1-f),\nabla F\rangle -\frac{2t}{1-f}\langle \nabla \log(1-f),\nabla q\rangle .
    \end{split}
  \end{equation}
  on $Q_{2R,T}$.
\end{lem}
\begin{proof}
  By the Bochner formula we have
  \begin{equation*}
    \begin{split}
      \Delta F=&2t|\Hess \log(1-f)|^2+2t\Ric (\nabla \log(1-f),\nabla \log(1-f))\\
      &+2t\langle \nabla \log(1-f),\nabla \Delta \log(1-f)\rangle .
    \end{split}
  \end{equation*}
  However, by \eqref{eqc}, 
  \begin{equation*}
    \begin{split}
      \Delta \log(1-f)=&-\frac{\Delta f}{1-f}-\frac{|\nabla f|^2}{(1-f)^2}=-fw-\frac{f_t+q+\hat A}{1-f}.
    \end{split}
  \end{equation*} 
  Therefore, we obtain
  \begin{equation*}
    \begin{split}
      \Delta F=&2t|\Hess \log(1-f)|^2+2t\Ric (\nabla \log(1-f),\nabla \log(1-f))\\
      &+\frac{2(1-f)}{t}F^2-2f\langle \nabla \log(1-f),\nabla F\rangle \\
      &+\frac{2}{1-f}(f_t+q+\hat A)F-\frac{2t}{1-f}\langle \nabla \log(1-f),\nabla (f_t)\rangle \\
      &-\frac{2t}{1-f}\langle \nabla \log(1-f),\nabla q\rangle -\frac{2t\hat A_f}{1-f}\langle \nabla \log(1-f),\nabla f\rangle .
    \end{split}
  \end{equation*}
  On the other hand, by the first equality of Lemma~\ref{lea},
  \begin{equation*}
    \begin{split}
      F_t=&-2th(\nabla \log(1-f),\nabla \log(1-f))\\
      &+2t\langle \nabla \log(1-f),\nabla ((\log(1-f))_t)\rangle +\frac{F}{t}\\
      =&-2th(\nabla \log(1-f),\nabla \log(1-f))\\
      &-\frac{2t}{1-f}\langle \nabla \log(1-f),\nabla (f_t)\rangle +\frac{2f_tF}{1-f}+\frac{F}{t}.
    \end{split}
  \end{equation*}
  Combining the above two equalities, we get 
  \begin{equation*}
    \begin{split}
      (\Delta -\partial_t)F\ge &\frac{2(1-f)}{t}F^2+2t(\Ric +h)(\nabla \log(1-f),\nabla \log(1-f))\\
      &+2\hat A_fF+\frac{2}{1-f}(q+\hat A)F-2f\langle \nabla \log(1-f),\nabla F\rangle \\
      &-\frac{2t}{1-f}\langle \nabla \log(1-f),\nabla q\rangle -\frac{F}{t}.
    \end{split}
  \end{equation*}
  The lemma follows from the assumption on bound of $\Ric +h$.
\end{proof}
 
\begin{rem}
  It is easy to see that we don't need any assumption on the Ricci tensor if geometric flow \eqref{eqb} is the Ricci flow, i.e., $h=-\Ric$.  
\end{rem}

\begin{proof}[The proof of Theorem~\ref{thc}]
  Choosing $\phi $ and $\eta $ as in the proof of Theorem~\ref{tha}. For any $T_1\in (0,T]$, let $(x_1,t_1)\in Q_{2R,T_1}$, at which $G(x,t)=\eta(x,t)F(x,t)$ attains its maximum, and without loss of generality, we can assume $G(x_1,t_1)>0$, and then $\eta(x_1,t_1)>0$ and $F(x_1,t_1)>0$. By Lemma~\ref{lec} and a similar argument as in the proof of Theorem~\ref{tha}, we have at $(x_1,t_1)$, 
  \begin{equation*}
    \begin{split}
      0\ge &\frac{2\eta (1-f)}{t_1}F^2-2\eta \left(K_1+K_2-\hat A_f-\frac{q}{1-f}-\frac{\hat A(f)}{1-f}\right)F\\
      &-2\eta f\langle \nabla \log(1-f),\nabla F\rangle -\frac{2t_1\eta }{1-f}\langle \nabla \log(1-f),\nabla q\rangle \\
      &-\frac{\eta F}{t_1}-\left(\frac{C_3(n)}{R^2}+\frac{C_4(n)}{R}\sqrt{K_1}+C_2K_2\right)F\\
      =&\frac{2\eta (1-f)}{t_1}F^2-2\eta \left(K_1+K_2-\hat A_f-\frac{q}{1-f}-\frac{\hat A(f)}{1-f}\right)F\\
      &+2f\langle \nabla \log(1-f),\nabla \eta \rangle F-\frac{2t_1\eta }{1-f}\langle \nabla \log(1-f),\nabla q\rangle \\
      &-\frac{\eta F}{t_1}-\left(\frac{C_3(n)}{R^2}+\frac{C_4(n)}{R}\sqrt{K_1}+C_2K_2\right)F.
    \end{split}
  \end{equation*}
  Multiplying both sides of the above inequality by $t_1\eta $, we have 
  \begin{equation*}
    \begin{split}
      0\ge &2(1-f)G^2-2t_1\eta \left(K_1+K_2-\hat A_f-\frac{q}{1-f}-\frac{\hat A(f)}{1-f}\right)G\\
      &+2t_1f\langle \nabla \log(1-f),\nabla \eta \rangle G-\frac{2t_1^2\eta^2}{1-f}\langle \nabla \log(1-f),\nabla q\rangle \\
      &-\eta G-t_1\left(\frac{C_3(n)}{R^2}+\frac{C_4(n)}{R}\sqrt{K_1}+C_2K_2\right)G.
    \end{split}
  \end{equation*}
  Noticing that $f\le 0$, by Young's inequality, 
  \begin{equation*}
    \begin{split}
      -2t_1f\langle \nabla \log(1-f),\nabla \eta \rangle G\le &-2ft_1^{\frac{1}{2}}\frac{|\nabla \eta |}{\eta^{\frac{1}{2}}}G^{\frac{3}{2}}\le -2ft_1\frac{\sqrt{C_1}}{R}G^{\frac{3}{2}}\\
      \le &(1-f)G^2+\frac{27t_1^2}{16}\frac{C_1^2}{R^4}\frac{f^4}{(1-f)^3}
    \end{split}
  \end{equation*}
  and
  \begin{equation*}
    \begin{split}
      \langle \nabla \log(1-f),\nabla q\rangle \le |\nabla q|w^{\frac{1}{2}}\le |q|w+\frac{|\nabla q|^2}{4|q|}.
    \end{split}
  \end{equation*}
  Combining the above three inequalities we have
  \begin{equation*}
    \begin{split}
      0\ge &(1-f)G^2-2t_1\eta \left(K_1+K_2-\hat A_f-\frac{q}{1-f}-\frac{\hat A(f)}{1-f}\right)G\\
      &-\frac{27t_1^2}{16}\frac{C_1^2}{R^4}\frac{f^4}{(1-f)^3}-\frac{2t_1\eta}{1-f}|q|G-\frac{t_1^2\eta2}{2(1-f)}\frac{|\nabla q|^2}{|q|}\\
      &-\eta G-t_1\left(\frac{C_3(n)}{R^2}+\frac{C_4(n)}{R}\sqrt{K_1}+C_2K_2\right)G.
    \end{split}
  \end{equation*}
  From $0\le \frac{1}{1-f}<1$ and $0<\frac{-f}{1-f}\le 1$, we see that 
  \begin{equation*}
    \begin{split}
      \hat A_f+\frac{\hat A(f)}{1-f}=&\frac{-f}{1-f}\hat A_f+\frac{1}{1-f}\left(\hat A_f+\hat A(f)\right)\\
      \ge &\frac{-f}{1-f}\min \left\{0,\min_{Q_{2R,T}}\hat A_f\right\}+\frac{1}{1-f}\min \left\{0,\min_{Q_{2R,T}}A'(u)\right\}\\
      \ge &\min \left\{0,\min_{Q_{2R,T}}\hat A_f,\min_{Q_{2R,T}}A'(u)\right\}\\
      =&-\kappa_{2R}.
    \end{split}
  \end{equation*}

  By $0<\eta(x_1,t_1)\le 1$, we get
  \begin{equation*}
    \begin{split}
      0\ge &(1-f)G^2-G-2t_1\left(K_1+K_2+\max_{Q_{2R,T}}(|q|-q)+\kappa_{2R}\right)G\\
      &-t_1\left(\frac{C_3(n)}{R^2}+\frac{C_4(n)}{R}\sqrt{K_1}+C_2K_2\right)G-\frac{27t_1^2}{16}\frac{C_1^2}{R^4}\frac{f^4}{(1-f)^3}-\frac{t_1^2}{2}\max_{Q_{2R,T}}\frac{|\nabla q|^2}{|q|}.
    \end{split}
  \end{equation*}
  Applying the quadratic formula and the inequality of arithmetic and geometric means 
  \[
  \frac{\sqrt{K_1}}{R}\le \frac{1}{2R^2}+\frac{K_1}{2},
  \]
  and noticing the fact $0\le \frac{-f}{1-f}<1$ again, we obtain 
  \begin{equation*}
    \begin{split}
      G\le &C_6(n)\left(1+t_1\left(\frac{1}{R^2}+H\right)\right),
    \end{split}
  \end{equation*}
  where $C_6(n)$ is a constant that depends only on $n$ and
  \[
  H=K_1+K_2+\max_{Q_{2R,T}}(|q|-q)+\max_{Q_{2R,T}}\frac{|\nabla q|}{\sqrt{|q|}}+\kappa_{2R}.
  \]
  Noticing that $d(x,x_0,T_1)\le R$ implies $\eta(x,T_1)=1$, we can get
  \begin{equation*}
    \begin{split}
      w(x,T_1)=\frac{F(x,T_1)}{T_1}\le \frac{G(x_1,t_1)}{T_1}\le C_6(n)\left(\frac{1}{T_1}+\frac{1}{R^2}+H\right).
    \end{split}
  \end{equation*}
  Since $T_1$ is arbitrary, and using $\sqrt{x+y}\le \sqrt{x}+\sqrt{y}$, we complete the proof.
\end{proof}

Similar to Corollary~\ref{coa}, when $(M,g(0))$ is a complete noncompact Riemannian manifold without boundary and $g(t)$ evolves by \eqref{eqb}, we can obtain a global estimate from Theorem~\ref{thc} by taking $R\to 0$. 
\begin{cor}\label{cob}
  Let $(M,g(t))$ be a complete solution to \eqref{eqb} for $t\in [0,T]$ and $(M,g(0))$ be a complete noncompact Riemannian manifold without boundary. Let $u$ be a positive solution to \eqref{eqa}. Suppose that there exist constants $L>0$ and $K\ge 0$, such that $u\le L$ and 
  \[
  \Ric \ge -K_1g,\quad h\ge -K_2g.
  \]
  Then we have 
  \begin{equation}\label{eqp}
    \begin{split}
      \frac{|\nabla u(x,t)|}{u(x,t)}\le \tilde C\left(\frac{1}{\sqrt{t}}+\sqrt{H}\right)\left(1+\log \frac{L}{u(x,t)}\right)
    \end{split}
  \end{equation}
  on $M\times [0,T]$, where $\tilde C$ as in Theorem~\ref{thc} and
  \[
  H=K_1+K_2+\sup_{M\times [0,T]}(|q|-q)+\sup_{M\times [0,T]}\frac{|\nabla q|}{\sqrt{|q|}}+\kappa .
  \]
\end{cor}

The following corollary gives a elliptic Harnack inequality by integrating the elliptic type gradient estimate \eqref{eqp} in space only. Unlike Corollary~\ref{coc}, this inequality can compare the function values at two spatial points at the same time, but inequality \eqref{eqt} cannot.
\begin{cor}\label{cod}
  Let $(M,g(t))$ be a complete solution to \eqref{eqb} for $t\in [0,T]$ and $(M,g(0))$ be a complete noncompact Riemannian manifold without boundary. Let $u$ be a positive solution to \eqref{eqa}. Suppose that there exist constants $L>0$ and $K\ge 0$, such that $u\le L$ and 
  \[
  \Ric \ge -K_1g,\quad h\ge -K_2g.
  \]
  Then for any $x_1,x_2\in M$, we have 
  \begin{equation}\label{equ}
    \begin{split}
      u(x_1,t)\ge u(x_2,t)^{\mathcal H}\exp ((1-\mathcal H)(1+\log L)).
    \end{split}
  \end{equation}
  in each $t\in [0,T]$. Here, $\mathcal H=\exp \left(\tilde C\left(\frac{1}{\sqrt{t}}+\sqrt{H}\right)d(x_1,x_2,t)\right)$,
  where $\tilde C$ as in Theorem~\ref{thc} and
  \[
  H=K_1+K_2+\sup_{M\times [0,T]}(|q|-q)+\sup_{M\times [0,T]}\frac{|\nabla q|}{\sqrt{|q|}}+\kappa .
  \] 
\end{cor}
\begin{proof}
  For any fixed $t$ and any $x_1,x_2\in M$, let $\zeta :[0,1]\to M$ is the geodesic of minimal length, which connecting $x_2$ and $x_1$, $\zeta(0)=x_2$ and $\zeta(1)=x_1$. Let $f=\log u$ and 
  \[
  l(s)=\log(1+\log L-f(\zeta(s),t)).
  \]
  By Corollary~\ref{cob} we have
  \begin{equation*}
    \begin{split}
      \frac{dl(s)}{ds}=&\frac{-\langle (\nabla f)(\zeta(s),t),\zeta'(s)\rangle }{1+\log L-f(\zeta(s),t)}\\
      \le &|\zeta'(s)|\cdot \frac{|\nabla f|(\zeta(s),t)}{1+\log L-f(\zeta(s),t)}\\
      \le &\tilde C|\zeta'(s)|\left(\frac{1}{\sqrt{t}}+\sqrt{H}\right).
    \end{split}
  \end{equation*}
  Integrating this inequality over $\zeta(s)$, we have
  \begin{equation*}
    \begin{split}
      \log \frac{1+\log L-f(x_1,t)}{1+\log L-f(x_2,t)}\le \tilde C\left(\frac{1}{\sqrt{t}}+\sqrt{H}\right)d(x_1,x_2,t).
    \end{split}
  \end{equation*}
  From this inequality, inequality \eqref{equ} can be obtained through a simple calculation.
\end{proof}

\section{Applications}\label{sec4}

We will give some applications of gradient estimates in section~2 and section~3 to some special equations. In some cases, we also take the geometric flow as the Ricci flow, i.e., $h=-Ric $ in \eqref{eqb}. 

\subsection{Applications of space-time gradient estimates}

In this subsection, we focus on applications of space-time gradient estimate for positive solutions. In this case, we see that $\lambda ,\Lambda ,\Sigma $ are not necessarily finite in \eqref{eqm}. For example, if we choose $A(u)=u^p$ for $p>0$, then $\lambda ,\Lambda ,\Sigma $ are all multiples of $\sup_{M\times [0,T]}u^{p-1}$ and they are not all zero if $p\ne 1$. But $\sup_{M\times [0,T]}u^{p-1}$ is not necessarily finite, unless $u$ is bounded. For the case that $q=0$ and $A(u)=u^p$, the reader can also refer to \cite{li2016Harnack,li2018li,zhao2016Gradient}. 

Naturally, for general unbounded positive function $u$, we want to know when $\lambda ,\Lambda $ and $\Sigma $ are finite. Since $u$ is a positive solution to \eqref{eqa}, $A(u)$ can be written as $u\cdot \frac{A(u)}{u}$, which is $A(u)=uH(u)$ for some $C^2$ function $H(u)$. As pointed by Q. Chen and the author in \cite[Remark~1.8]{chen2018Li}, $\lambda ,\Lambda ,\Sigma <+\infty $ implies
\[
|H(u)|\le C_0\log u,\quad \mbox{as}\ u\to +\infty.
\]
When $A(u)=au\log u$, A direct computation yields $\Sigma =0$ and 
\begin{equation*}
  \begin{cases}
    \lambda_{2R}=\lambda \equiv 0,\quad \Lambda_{2R}=\Lambda \equiv a,\quad &\mbox{if}\ a\ge 0,\\
    \lambda_{2R}=\lambda \equiv -a ,\quad \Lambda_{2R}=\Lambda \equiv 0,\quad &\mbox{if}\ a\le 0.
  \end{cases}
\end{equation*} 
Therefore, we can obtain that local and global gradient estimates for positive solutions of the equation 
\begin{equation}\label{eqr}
  \begin{split}
    (\Delta -q(x,t)-\partial_t)u(x,t)=au(x,t)\log (u(x,t)),\quad a\in \mathbb R
  \end{split}
\end{equation}
from Theorem~\ref{tha} and Corollary~\ref{coa}. 

\begin{rem}
By the asymptotic behavior of $H$, we can find many examples that satisfy $\lambda ,\Lambda ,\Sigma <+\infty $. Such as $H(u)=\frac{P_k(u)}{P_l(u)}$, where $P_k, P_l$ are polynomials of degree $k,l$, respectively and $k<l$. Hence we can obtain gradient estimates for positive solution of the following series of equations
\[
(\Delta -q(x,t)-\partial_t)u(x,t)=au(x,t)\frac{P_k(u(x,t))}{P_l(u(x,t))},\quad a\in \mathbb R,\ k<l.
\]
For instance, if we choose $k=0, l=1$, so $A(u)=\frac{u}{u+1}$, then $\lambda =\frac{1}{4}, \Lambda =0$ and $\Sigma =\frac{8}{27}$. 
\end{rem}

On the other hand, as mentioned in Remark~\ref{rea}, if $A(u)=au\log u$, then $\Sigma =0$. In addition, we take $h=-\Ric $. In this case, we don't need the assumption on the bound $|\nabla h|$ since the contracted second Bianchi identity. At this time, when $\alpha \to 1$, we can also let $\varepsilon \to 0$ in local estimate \eqref{eql}, and then we are arriving at
\begin{cor}\label{coe}
  Let $(M,g(0))$ be a complete Riemannian manifold, and let $g(t)$ evolves by the Ricci flow for $t\in [0,T]$. Suppose that there exist constants $K_2, \theta \ge 0$ such that 
  \[
  0\le \Ric \le K_2g
  \]
  and 
  \[
  \Delta q\le \theta_{2R}
  \]
  on $Q_{2R,T}$. If $u$ is a positive solution to \eqref{eqr}. Then on $Q_{R,T}$, we have 
  \begin{enumerate}
  \item for $a\ge 0$, 
  \begin{equation*}
    \begin{split}
      \frac{|\nabla u(x,t)|^2}{u^2(x,t)}-\frac{u_t(x,t)}{u(x,t)}-q(x,t)-\frac{A(u(x,t))}{u(x,t)}
      \le \frac{n}{t}+(C+n)K_2+\sqrt{n\theta }+\frac{n}{2}a;
    \end{split}
  \end{equation*}
  \item for $a<0$, 
  \begin{equation*}
    \begin{split}
      \frac{|\nabla u(x,t)|^2}{u^2(x,t)}-\frac{u_t(x,t)}{u(x,t)}-q(x,t)-\frac{A(u(x,t))}{u(x,t)}
      \le \frac{n}{t}+(C+n)K_2+\sqrt{n\theta }-na,
    \end{split}
  \end{equation*}
  \end{enumerate}
  where $C$ as in Theorem~\ref{tha}.
\end{cor}

From the above local estimate, we have immediately
\begin{cor}\label{cof}
  Let $(M,g(0))$ be a complete noncompact Riemannian manifold without boundary, and let $g(t)$ evolves by the Ricci flow for $t\in [0,T]$. Suppose that there exist constants $K_2, \theta \ge 0$ such that 
  \[
  0\le \Ric \le K_2g
  \]
  and 
  \[
  \Delta q\le \theta .
  \]
  If $u$ is a positive solution to \eqref{eqr}. Then we have  
  \begin{equation*}
    \begin{split}
      \frac{|\nabla u(x,t)|^2}{u^2(x,t)}-\frac{u_t(x,t)}{u(x,t)}-q(x,t)-\frac{A(u(x,t))}{u(x,t)}
      \le \frac{n}{t}+C''\left(K_2+\sqrt{\theta }+|a|\right)
    \end{split}
  \end{equation*}
  on $M\times [0,T]$, where $C''$ is a constant that depends only on $n$.
\end{cor}

When the manifold is closed, we also have 
\begin{cor}\label{cog}
  Let $(M,g(t))$ be a closed Riemannian manifold, where $g(t)$ evolves by the Ricci flow for $t\in [0,T]$ and satisfies
  \[
  0\le \Ric \le K_2g.
  \]
  If $u$ is a positive solution to the equation
  \[
  (\Delta -q(x,t)-\partial_t)u(x,t)=au(x,t)\log (u(x,t)),
  \]
  and $q(x,t)$ satisfies
  \[
  \Delta q\le \theta .
  \]
  Then we have 
  \begin{equation*}
    \begin{split}
      \frac{|\nabla u(x,t)|^2}{u^2(x,t)}-\frac{u_t(x,t)}{u(x,t)}-q(x,t)-\frac{A(u(x,t))}{u(x,t)}\le \frac{n}{2t}+nK_2+\sqrt{n\theta }+\frac{n}{2}|a|
    \end{split}
  \end{equation*}
  on $M\times (0,T]$.
\end{cor}

\subsection{Applications of elliptic type gradient estimates}

Now we give some applications of elliptic type gradient estimates for bounded positive solutions. Since we are dealing with bounded positive solutions, $A(u)$ that satisfies the conditions $\kappa <+\infty $ is easy to find. 

We will consider that elliptic type gradient estimates for bounded positive solutions of the equation
\begin{equation}\label{eqs}
  \begin{split}
    (\Delta -q(x,t)-\partial_t)u(x,t)=au(x,t)^\beta ,\quad a\in \mathbb R,\ \beta \in (-\infty ,0]\cup [1,+\infty).
  \end{split}
\end{equation}
In order not to be redundant, we only give the global estimate here, and the local one is omitted. 
\begin{cor}\label{coh}
  Let $(M,g(t))$ be a complete solution to \eqref{eqb} for $t\in [0,T]$ and $(M,g(0))$ be a complete noncompact Riemannian manifold without boundary. Let $u$ be a positive solution to \eqref{eqs}. Suppose that there exist constants $L>0$ and $K\ge 0$, such that $u\le L$ and 
  \[
  \Ric \ge -K_1g,\quad h\ge -K_2g.
  \]
  Then on $M\times [0,T]$, we have 
  \begin{equation*}
    \begin{split}
      \frac{|\nabla u(x,t)|}{u(x,t)}\le \tilde C\left(\frac{1}{\sqrt{t}}+\sqrt{H'}\right)\left(1+\log \frac{L}{u(x,t)}\right),
    \end{split}
  \end{equation*}
  where $\tilde C$ as in Theorem~\ref{thc} and
  \[
  H'=K_1+K_2+\sup_{M\times [0,T]}(|q|-q)+\sup_{M\times [0,T]}\frac{|\nabla q|}{\sqrt{|q|}}+\kappa_0
  \]
  with 
  \begin{equation*}
    \kappa_0=
    \begin{cases}
      \frac{\sign a-1}{2}a\beta L^{\beta -1},\ &\mbox{if}\ a\in \mathbb R,\ \beta \ge 1,\\
      0,\ &\mbox{if}\ a\le 0,\ \beta \le 0,\\
      a(1-\beta )\left(\inf_{M\times [0,T]}u(x,t)\right)^{\beta -1} ,\ &\mbox{if}\ a\ge 0,\ \beta \le 0.
    \end{cases}
  \end{equation*}
  Here, $\sign a$ is the sign function, which is $1, 0, -1$ if $a>0, =0, <0$, respectively.
\end{cor}
\begin{proof}
  From Corollary~\ref{cob}, we just have to compute $\kappa $. By the definition, we have 
\begin{equation*}
  \begin{split}
    \kappa =&-\min \left\{0,\inf_{M\times [0,T]}(a\beta u^{\beta -1}),\inf_{M\times [0,T]}(a(\beta -1)u^{\beta -1})\right\}\\
    =&
    \begin{cases}
      0,\ &\mbox{if}\ a\ge 0,\ \beta \ge 1,\\
      -a\beta L^{\beta -1},\ &\mbox{if}\ a\le 0,\ \beta \ge 1,\\
      0,\ &\mbox{if}\ a\le 0,\ \beta \le 0,\\
      a(1-\beta )\left(\inf_{M\times [0,T]}u(x,t)\right)^{\beta -1} ,\ &\mbox{if}\ a\ge 0,\ \beta \le 0.
    \end{cases}
  \end{split}
\end{equation*}
Therefore, we obtain the corollary.
\end{proof}

In particular, when $q(x,t)=const.$, the term $qu(x,t)$ can be combined by $au(x,t)^\beta $, so we get
\begin{cor}\label{coi}
  Let $(M,g(t))$ be a complete solution to \eqref{eqb} for $t\in [0,T]$ and $(M,g(0))$ be a complete noncompact Riemannian manifold without boundary. Let $u$ be a positive solution to 
  \[
  (\Delta -\partial_t)u(x,t)=au(x,t)^\beta ,\quad a\in \mathbb R,\ \beta \in (-\infty ,0]\cup [1,+\infty).
  \]
  Suppose that there exist constants $L>0$ and $K\ge 0$, such that $u\le L$ and 
  \[
  \Ric \ge -K_1g,\quad h\ge -K_2g.
  \]
  Then on $M\times [0,T]$, we have 
  \begin{equation*}
    \begin{split}
      \frac{|\nabla u(x,t)|}{u(x,t)}\le \tilde C\left(\frac{1}{\sqrt{t}}+\sqrt{K_1+K_2+\kappa_0}\right)\left(1+\log \frac{L}{u(x,t)}\right),
    \end{split}
  \end{equation*}
  where $\tilde C$ as in Theorem~\ref{thc} and 
  \begin{equation*}
    \kappa_0=
    \begin{cases}
      \frac{\sign a-1}{2}a\beta L^{\beta -1},\ &\mbox{if}\ a\in \mathbb R,\ \beta \ge 1,\\
      0,\ &\mbox{if}\ a\le 0,\ \beta \le 0,\\
      a(1-\beta )\left(\inf_{M\times [0,T]}u(x,t)\right)^{\beta -1} ,\ &\mbox{if}\ a\ge 0,\ \beta \le 0.
    \end{cases}
  \end{equation*}
  Here, $\sign a$ is the sign function, which is $1, 0, -1$ if $a>0, =0, <0$, respectively.
\end{cor}

\begin{rem}
  For each of these specific equations that appear in this section, we also have the corresponding Harnack inequality, which we will not write them all down here. 
\end{rem}


\bibliographystyle{abbrv}
\bibliography{gugbibfile}

\end{document}